\documentclass[a4paper,11pt]{amsart}
\usepackage[latin1]{inputenc}
\usepackage[T1]{fontenc}
\usepackage[english]{babel}
\usepackage[foot]{amsaddr}
\usepackage{amsfonts, amsmath, amssymb, mathrsfs}
\usepackage{dsfont}
\usepackage{amsthm}
\makeatletter
\@namedef{subjclassname@2020}{%
  \textup{2020} Mathematics Subject Classification}
\makeatother
\usepackage{mathtools}
\usepackage{braket}
\usepackage{graphicx,epstopdf}
\usepackage{caption}
\usepackage{tikz}
\usetikzlibrary{arrows,calc,matrix}
\usepackage{soul}
\usepackage{hyperref}
\theoremstyle{plain}
\newtheorem{Theorem}{Thm}[section]

\newtheorem{Lem}[Theorem]{Lemma}

\theoremstyle{definition}
\newtheorem{Exa}[Theorem]{Example}

\newcommand{\C}{\mathbb{C}}

\renewcommand\P{\mathbb{P}}
\newcommand{\R}{\mathbb{R}}

\newcommand{\ii}{\mathrm{i}}
\newcommand{\id}{\mathds{1}}

\begin{document}
\selectlanguage{English}
\title{Kippenhahn's construction revisited}
\author{Stephan Weis\textsuperscript{1}}
\address{{}\textsuperscript{1} Wald-Gymnasium Berlin, 
Germany,
e-mail: \texttt{maths@weis-stephan.de},\newline
ORCID: \texttt{0000-0003-1316-9115}}
\begin{abstract}
Kippenhahn discovered that the numerical range of a complex square 
matrix is the convex hull of a plane real algebraic curve. Here, we 
present an example of a convex set, which has a similar algebraic 
description as the numerical range, whereas the analogue of 
Kippenhahn's construction fails regarding isolated, singular points 
of the curve. This example prompted us to carefully review 
Kippenhahn's assertion and to highlight aspects of a complete proof
that was achieved with methods of convex geometry and real algebraic 
geometry.
\end{abstract}
\date{June 11th, 2023}
\maketitle
%
%
%
%
\section{Introduction}
Let $A$ be an $n\times n$ matrix with complex coefficients.
The \emph{numerical range}, also known as \emph{field of values},
of $A$ is the subset of the complex plane $\C$ defined by
\[
W(A)=\{\langle \eta\mid A\eta\rangle : 
\langle \eta\mid \eta\rangle=1, \eta\in\C^n\}\,,
\]
where
$\langle \eta\mid \xi\rangle=\eta_1\overline{\xi_1}+\eta_2\overline{\xi_2}+\dots+\eta_n\overline{\xi_n}$ 
denotes the inner product of two vectors $\eta=(\eta_1,\eta_2,\dots,\eta_n)^T$
and $\xi=(\xi_1,\xi_2,\dots,\xi_n)^T$ in $\C^n$. 
The numerical range $W(A)$ is invariant under unitary similarity. It is used to 
study spectra and norms of matrices \cite{HornJohnson1991} and of operators 
\cite{GauWu2021,Jefferies2004}. 
\par
The shape of the numerical range is well understood \cite{Keeler-etal1997}
for matrices of sizes $n=2$ and $n=3$. The shape remains a research topic 
in matrix theory \cite{ChienNakazato2012,Bebiano-etal2021,JiangSpitkovsky2022} 
and quantum information theory \cite{Gawron-etal2010} for matrices of sizes 
$n\geq 4$. Among the basic key-theorems is the 
Toeplitz-Hausdorff theorem \cite{Toeplitz1918,Hausdorff1919} from the years 
1918/19, which asserts that $W(A)$ is a convex set, and 
Kippenhahn's assertion \cite[Sec.~3, Nr.~10]{Kippenhahn1951} from 1951, which 
states that $W(A)$ is the convex hull of a plane real algebraic curve $C_A$. 
The curve $C_A$ has been called \emph{boundary generating curve} 
\cite{Kippenhahn1951,ChienNakazato2010}, and was renamed to 
\emph{Kippenhahn curve} 
\cite{Daepp-etal2018,GauWu2021,Bebiano-etal2021,JiangSpitkovsky2022}
recently. 
\par
In this paper, we take a closer look at Kippenhahn's proof that the curve $C_A$ 
is included in $W(A)$. We illustrate Kippenhahn's construction of $C_A$ and his
assertion with a $3\times 3$ matrix $A$ in Section~\ref{sec:constructKC}. 
In Section~\ref{sec:revealingexample} we present a convex set $W$\! for which 
a similar construction is feasible, but the curve analogous to $C_A$ has isolated, 
singular points outside of $W$\!. These examples are supported by the computer 
algebra system \texttt{Wolfram Mathematica}. Despite this flaw, Kippenhahn's 
proof is still advocated \cite[Thm.~1.3]{GauWu2021}, 
\cite[Thm.~13.4]{Daepp-etal2018} without paying attention that it is incomplete. 
An exception is the paper \cite{ChienNakazato2010} where Chien and Nakazato 
analyze singular points of the numerical range.
\par
During the last fifteen years or so, we witnessed an increasing research 
activity at the crossroads of convex geometry and real algebraic geometry, 
a field called \emph{convex algebraic geometry}
\cite{HeltonVinnikov2007,Henrion2010,Sanyal-etal2011,Netzer2012,PlaumannVinzant2013,
Blekerman-etal2013,Sinn2015,Scheiderer2018}. 
To our surprise, Sinn \cite[Example~3.15]{Sinn2015} had proven an assertion 
similar to Kippenhahn's already in 2014. By translating Sinn's result 
to the numerical range, and by showing that all singular points of $C_A$ lie 
inside $W(A)$, we obtained a complete proof of Kippenhahn's assertion 
\cite[Thm.~1.1]{PSW2021}. 
\par
The set $W$ mentioned above is not the numerical range of a matrix since the 
analogue of Kippenhahn's assertion fails for $W$\!. This follows also from 
the fact that the dual convex set $W^\ast$ to $W$ is bounded by the 
\emph{Fermat curve} $x_1^6+x_2^6=1$. All lines 
in $\R^2$ through this curve intersect it in two points instead of six, see 
Fig.~\ref{fig:television6}b). In other words, $W^\ast$ is not 
\emph{rigidly convex} \cite{HeltonVinnikov2007} and therefore it cannot be 
the dual convex set to the numerical range of a matrix 
\cite{Henrion2010,HeltonSpitkovsky2012}. This is, of course, a peculiarity of 
the finite dimensionality. Any bounded convex nonempty subset of the plane is 
the numerical range of a bounded linear operator on a Hilbert space 
\cite{Pollack1974}.
\par
Section~\ref{sec:Kippenhahnright} highlights some aspects of the proof
that the curve $C_A$ is included in $W(A)$. The preliminary 
Sec.~\ref{sec:dualpicture} changes perspective from points outside of 
$W(A)$ to lines crossing the dual convex set.
\par
In this paper, we do not address the converse part of Kippenhahn's assertion, 
that $W(A)$ is included in the convex hull of $C_A$. This follows 
from \cite[Cor~3.14]{Sinn2015}, and is adapted to the setting of the numerical 
range in Thm.~4.5 and Sec.~6 in \cite{PSW2021}.
\par
%
%
\section{Construction of the Kippenhahn Curve}
\label{sec:constructKC}
Following Kippenhahn \cite{Kippenhahn1951}, we construct the curve $C_A$ 
in three steps along the drawings in Fig.~\ref{fig:NRdeg6}. For the sake 
of clarity, we consider only irreducible curves defined by irreducible 
polynomials in the sequel. A general reference for algebraic curves is 
Fischer \cite{Fischer2001}. 
\par
We identify $\C=\R\oplus\ii\R\cong\R^2$, where $\ii$ is the imaginary unit, 
and we view the numerical range $W(A)$ as a subset of $(\R^2)^\ast$, the 
dual vector space to $\R^2$. Writing $A=K+\ii L$, where $K,L\in\C^{n\times n}$
are hermitian matrices, we identify the numerical range $W(A)$ with the set
\begin{equation}\label{eq:NR}
W=\{( \langle \eta\mid K\eta\rangle, \langle \eta\mid L\eta\rangle) 
: \langle \eta\mid \eta\rangle=1, \eta\in\C^n\} \subset (\R^2)^\ast\,.
\end{equation}
\par
\textbf{Step 1.}
The \emph{support function} \cite{Schneider2014} of the numerical range
$W\subset(\R^2)^\ast$ at the direction $x\in\R^2$ is
\[
h(x)=\min_{y\in W}\langle x,y\rangle\,,
\]
where the pairing of $x=(x_1,x_2)^T\in\R^2$ and 
$y=(y_1,y_2)\in(\R^2)^\ast$ is given by $\langle x,y\rangle=x_1y_1+x_2y_2$. 
The number $h(x)$ is the signed distance of the origin $(0,0)\in(\R^2)^\ast$ 
from the \emph{supporting line} with inner normal vector $x$,
\begin{equation}\label{eq:supplineaffine}
\{y\in(\R^2)^\ast : \langle x,y\rangle=h(x)\}\,.
\end{equation}
This means that $x$ is perpendicular to this line and points into the 
half-space $\{y\in(\R^2)^\ast : \langle x,y\rangle\geq h(x)\}$, which contains
$W(A)$. Toeplitz \cite{Toeplitz1918} showed that $h(x)$ is the smallest 
eigenvalue of the hermitian matrix $x_1K+x_2L$. Let 
\begin{equation}\label{eq:KL}
K=
\left(
\begin{array}{ccc}
  0 & -1 & 0 \\
 -1 &  0 & 1 \\
  0 &  1 & 0 \\
\end{array}
\right)
\quad\text{and}\quad
L=
-\frac{1}{4}\left(
\begin{array}{ccc}
  1 & 2 & -4 \\
  2 & 1 &  2 \\
 -4 & 2 &  1 \\
\end{array}
\right)\,.
\end{equation}
Some supporting lines for this example are depicted in Fig.~\ref{fig:NRdeg6}a).
Note that $h(x)$ is negative as $x_1K+x_2L$ has a negative eigenvalue for 
nonzero $x\in\R^2$. The rounded triangle in the center of 
Fig.~\ref{fig:NRdeg6}a), without lines crossing, depicts the numerical range.
\par
\begin{figure}%
\parbox[b]{5mm}{\raggedleft a)}
\parbox[b]{5cm}{%
\includegraphics[width=5cm]{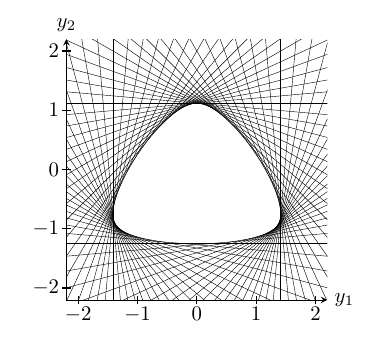}}%
\hspace{5mm}%
\parbox[b]{5mm}{\raggedleft b)}
\parbox[b]{5cm}{%
\includegraphics[width=5cm]{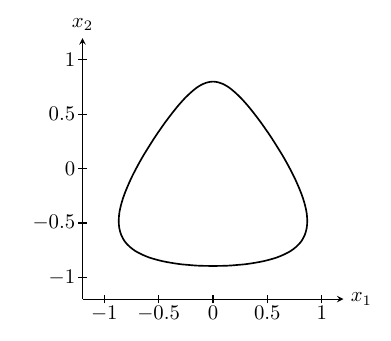}}%
\vskip1\baselineskip
\parbox[b]{5mm}{\raggedleft c)}
\parbox[b]{5cm}{%
\includegraphics[width=5cm]{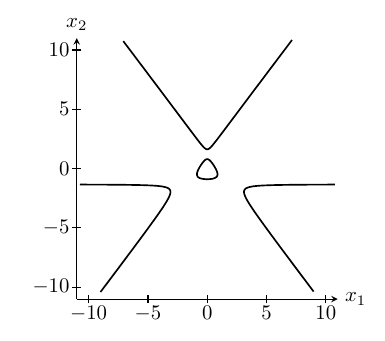}}%
\hspace{5mm}%
\parbox[b]{5mm}{\raggedleft d)}
\parbox[b]{5cm}{%
\includegraphics[width=5cm]{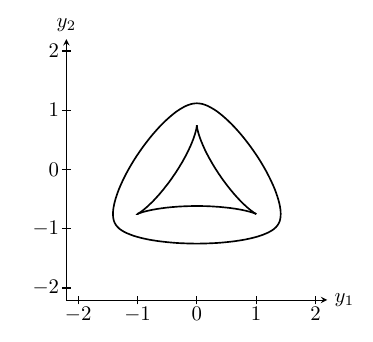}}%
\caption{%
a) Supporting lines of the numerical range $W(A)$ with matrices $A=K+\ii L$ 
from equation~\eqref{eq:KL}. 
b) and c) The real, affine part of the algebraic curve $D$ that contains the 
supporting lines of $W(A)$.
d) The Kippenhahn curve $C_A$.}
\label{fig:NRdeg6}
\end{figure}
\textbf{Step 2.}
To obtain an algebraic equation for the set of supporting lines, we use
a one-to-one correspondence between the points in the complex projective 
plane $\P^2$ and lines in the dual projective plane $(\P^2)^\ast$. The 
\emph{polar} of a point $(x_0:x_1:x_2)\in\P^2$ is the line 
\[
\{(y_0:y_1:y_2)\in(\P^2)^\ast \mid x_0y_0+x_1y_1+x_2y_2=0\}\,.
\]
Conversely, the point $(x_0:x_1:x_2)$ is the \emph{pole} of this line.
The line is a \emph{real line} if its pole can be written with real numbers 
$x_0,x_1,x_2$. (In this paper we use poles and polars with respect to the 
quadric $\{x\in\P^2:x_0^2+x_1^2+x_2^2=0\}$.) We also use the analogous 
one-to-one correspondence between points in $(\P^2)^\ast$ and lines in 
$\P^2$ that comes with the biduality of $\P^2=((\P^2)^\ast)^\ast$. 
We will often abuse notation and identify a line with its pole.
Furthermore, we employ the embedding of the affine plane $\C^2$ into the 
projective plane $\P^2$ \emph{via} $(x_1,x_2)^T\mapsto(1:x_1:x_2)$ and 
the embedding of the dual affine plane $(\C^2)^\ast$ into the dual 
projective plane $(\P^2)^\ast$ \emph{via} $(y_1,y_2)\mapsto(1:y_1:y_2)$. 
\par
For all nonzero $(x_1,x_2)^T\in\R^2$, the pole of the supporting line in equation~\eqref{eq:supplineaffine} is 
\[
\ell=(-h(x):x_1:x_2)\in\P^2\,.
\]
The line $\ell$ satisfies the algebraic equation $p(\ell)=0$, where 
$p\in\R[x_0,x_1,x_2]$ is the homogeneous polynomial defined by the determinant
\[
p=\det(x_0\id+x_1K+x_2L)\,,
\]
and where $\id\in\C^{n\times n}$ is the identity matrix. This means that 
$\ell$ lies on the curve 
\[
D=\{x\in\P^2:p(x)=0\}\,.
\]
The \emph{real, affine part} of this curve is
$\{(x_1,x_2)^T\in\R^2 \mid (1:x_1:x_2)\in D\}$.
\par
The polynomial $p$ in the example \eqref{eq:KL} is provided in
Sec.~\ref{app:dualcurve}. The real, affine part of $D$ has a connected component 
of the shape of a rounded triangle, see Fig.~\ref{fig:NRdeg6}b), and three 
peripheral, unbounded, connected components outside the rounded triangle, 
see Fig.~\ref{fig:NRdeg6}c). The inner rounded triangle is the set of supporting 
lines of the numerical range $W$. The peripheral components do not contain any 
supporting lines of $W$.
\par
\textbf{Step 3.}
The final step is to pass from the curve $D$ to its dual curve. A point 
$x\in D$ is a \emph{regular point} of $D$ if at least one of the partial 
derivatives $\partial_{x_0}p$, $\partial_{x_1}p$, $\partial_{x_2}p$ is 
nonzero at $x$ (the polynomial $p$ is assumed to be irreducible). 
Otherwise, $x$ is a \emph{singular point} of $D$. If $x$ is a regular 
point, then the tangent line to $D$ at $x$ is
\[
(\partial_{x_0}p(x):\partial_{x_1}p(x):\partial_{x_2}p(x))\in(\P^2)^\ast\,.
\]
The \emph{dual curve} $D^\vee$ to $D$ is the closure in the norm topology
of the set of tangent lines to the curve $D$ at all regular points, see 
\cite[Sec.~5.1 and 3.6]{Fischer2001}. This closure comprises exactly the
set of all tangents to the curve $D$, at regular and singular points. 
It is also possible to replace the norm closure with the Zariski closure 
\cite[Thm.~2.33]{Mumford1976}.
\par
One can use Gröbner bases \cite{Cox-etal2015} to compute the algebraic equation 
for the dual curve $D^\vee$\!. In the above example, we have 
$D^\vee=\{y\in(\P^3)^\ast : q(y)=0\}$, where $q\in\R[y_0,y_1,y_2]$ is an
irreducible, homogeneous polynomial of degree six. The polynomial and the code 
from \texttt{Wolfram Mathematica} used for its calculation are provided in 
Sec.~\ref{app:dualcurve}. The real, affine part 
\[
C_A=\{(y_1,y_2)\in(\R^2)^\ast \mid (1:y_1:y_2)\in D^\vee\}
\]
of the dual curve is the Kippenhahn curve.
\par
Fig.~\ref{fig:NRdeg6}d) depicts the Kippenhahn curve $C_A$ of the example 
\eqref{eq:KL}. Here, $C_A$ has two connected components, a rounded triangle 
and a curve inside with three cusps \cite{Kippenhahn1951,Keeler-etal1997}. The 
rounded triangle contains the tangent lines at the central, rounded triangle 
in the real, affine part of $D$, shown in Fig.~\ref{fig:NRdeg6}b) and 
Fig.~\ref{fig:NRdeg6}c). The inner curve of $C_A$ comprises the tangent lines 
at the three peripheral components of the real, affine part of $D$ shown in 
Fig.~\ref{fig:NRdeg6}c).
\par
\begin{Theorem}[Kippenhahn's Assertion]\label{thm:Kippenhahn}
Let $K,L\in\C^{n\times n}$ be hermitian matrices. Then the numerical range 
$W=W(A)$ of $A=K+\ii L$ is the convex hull of the Kippenhahn curve $C_A$. 
\end{Theorem}
We refer to \cite[Sec.~6]{PSW2021} for a proof of Thm.~\ref{thm:Kippenhahn}. 
We will see in Sec.~\ref{sec:Kippenhahnright} why the theorem is unaffected 
by singular points, contrary to the example in the following 
section.
%
%
%
\begin{figure}[t]%
\parbox[b]{5mm}{\raggedleft a)}
\parbox[b]{5cm}{%
\includegraphics[width=5cm]{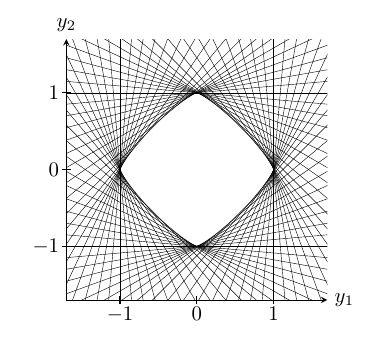}}%
\hspace{5mm}%
\parbox[b]{5mm}{\raggedleft b)}
\parbox[b]{5cm}{%
\includegraphics[width=5cm]{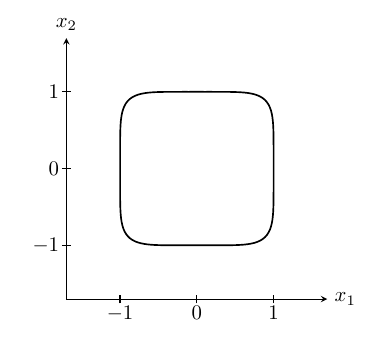}}%
\vskip1\baselineskip
\parbox[b]{5mm}{\raggedleft c)}
\parbox[b]{5cm}{%
\includegraphics[width=5cm]{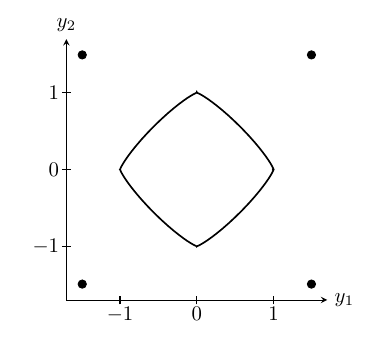}}%
\hspace{5mm}%
\parbox[b]{5mm}{\raggedleft d)}
\parbox[b]{5cm}{%
\includegraphics[width=5cm]{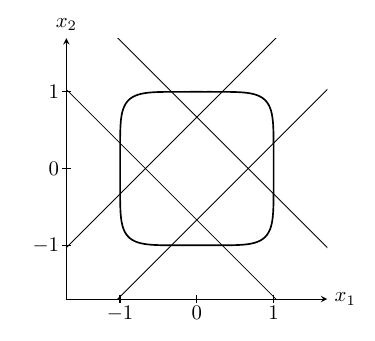}}%
\caption{%
a) Supporting lines of the convex set $W$ defined in 
equation~\eqref{eq:htele}.
b) Real, affine part $x_1^6+x_2^6=1$ of the algebraic curve $D$ that 
contains the supporting lines of $W$.
c) The real, affine part of the dual curve $D^\vee$, with isolated, 
singular points outside of $W$.
d) Real (double) tangents of the curve $D$ intersect the set 
$x_1^6+x_2^6<1$.}
\label{fig:television6}
\end{figure}
\section{A Revealing Example}
\label{sec:revealingexample}
Here, we present a convex set, which has a similar algebraic description 
as the numerical range, whereas the analogue of Thm.~\ref{thm:Kippenhahn} 
fails.
\par
\textbf{Step 1.} 
Using the notation from the prior section, we define a compact, convex set $W\subset(\R^2)^\ast$ in terms of its support function
\begin{equation}\label{eq:htele}
h(x)=-\sqrt[6]{x_1^6+x_2^6}\,,
\quad x\in\R^2\,.
\end{equation}
Some supporting lines of $W$ are depicted in Fig.~\ref{fig:television6}a).
The rounded rhomb in the center of Fig.~\ref{fig:television6}a), without 
lines crossing, depicts the convex set $W$.
\par
\textbf{Step 2.} 
For every nonzero $x\in\R^2$, the supporting line of $W$ with inner normal 
vector $x$ has the pole $\ell=(-h(x):x_1:x_2)\in\P^2$, which lies on the 
curve $D=\{x\in\P^2:p(x)=0\}$, where 
\[
p=x_0^6-x_1^6-x_2^6\,.
\]
Fig.~\ref{fig:television6}b) depicts a rounded square, which is the real, 
affine part of $D$, and which is the set of all supporting lines of $W$.
\par
\textbf{Step 3.}
The dual curve to $D$ is the variety $D^\vee=\{y\in(\P^3)^\ast : q(y)=0\}$ 
of an irreducible, homogeneous polynomial $q\in\R[y_0,y_1,y_2]$ of degree $30$,
see Sec.~\ref{app:dualcurve} for the polynomial and the code from 
\texttt{Wolfram Mathematica} used for the calculation. The curve $D^\vee$ has 
$228$ singular points, eight of which are real. These are $(1,0)$, $(0,1)$, 
$(-1,0)$, $(0,-1)$,  
\begin{equation}\label{eq:singularpointsCA}
(\omega,\omega), \quad
(\omega,-\omega), \quad
(-\omega,\omega), \quad\text{and}\quad
(-\omega,-\omega)\,,
\end{equation}
where $\omega=\left((11+5 \sqrt{5})/2)\right)^{1/6}\approx1.49\,$. The real, 
affine part of $D^\vee$ consists of a rounded rhomb, which contains the first
four singular points, and of the four isolated, singular points, which are 
depicted in Fig.~\ref{fig:television6}c). The analogue of 
Thm.~\ref{thm:Kippenhahn} fails because the isolated singular points lie 
outside of $W$. Their polars are real double tangents of the curve $D$, 
depicted in Figure~\ref{fig:television6}d). For example, the polar of 
$(1:\omega:\omega)\in D^\vee$ is tangent to $D$ at the pair of complex 
conjugate points 
\[
\frac{1}{2\omega}
(-1\pm\ii\alpha,-1\mp\ii\alpha)^T\,,
\qquad \alpha=\sqrt{5+2 \sqrt{5}}\,.
\]
%
%
\section{The Dual Picture of Points Outside the Numerical Range}
\label{sec:dualpicture}
Some singular points of the dual curve $D^\vee$ lie outside of $W$ in the 
example of Sec.~\ref{sec:revealingexample}. Here we capture the location 
of points in- or outside of $W$ in terms of whether their polars intersect
the dual convex set to $W$. 
\par
We consider a compact, convex subset $W$ of $(\R^2)^\ast$ that contains the 
origin $(0,0)\in(\R^2)^\ast$ as an interior point. The \emph{dual convex set} 
to $W$ is defined by
\[
W^\ast=\{(x_1,x_2)^T : 1+x_1y_1+x_2y_2\geq 0 \,\,\forall y\in W \}\,.
\]
The set $W^\ast$ is a compact, convex subset of $\R^2$ that contains the 
origin $(0,0)^T\in\R^2$ as an interior point, and we have $(W^\ast)^\ast=W$. 
See \cite[Thm.~1.6.1]{Schneider2014} for these statements.
\par
If the origin is not an interior point $W$\!, then conic duality 
instead of convex duality makes the ideas of the following 
Sec.~\ref{sec:Kippenhahnright} work \cite[Sec.~5]{PSW2021}. 
\begin{Lem}\label{Lem:WS}
A real point in $(\P^2)^\ast$ lies outside of $W$ if and only if 
its polar intersects the interior of $W^\ast$.
\end{Lem}
\begin{proof}
Since we associated the support function $h$ with $W$, we prove the equivalent 
dual statement that a real line in $(\P^2)^\ast$ intersects the interior of $W$ 
if and only if its pole lies outside of $W^\ast$.
\par
We consider a line in $(\R^2)^\ast$, which is not incident with the origin. By 
this we mean a real line in $(\P^2)^\ast$, such that its pole 
$m=(x_0:x_1:x_2)\in\P^2$ has the affine coordinates 
$(\frac{x_1}{x_0},\frac{x_2}{x_0})^T$ for some $x_0\neq 0$. If we fix a nonzero 
tuple $x=(x_1,x_2)^T\in\R^2$, then the supporting line of $W$ with inner normal 
vector $x$ has the pole $\ell=(-h(x):x_1:x_2)$ with the affine coordinates 
$\,-(x_1,x_2)^T/h(x)\,$. These real tuples parametrize the boundary of the 
dual convex set $W^\ast$ if $(x_1,x_2)^T$ varies in a closed curve about the 
origin in $\R^2$, according 
to \cite[Sec.~1.7.2, Eq.~(1.52)]{Schneider2014}. 
\par
Let $x_0>0$. Then the polar of $m=(x_0:x_1:x_2)$ lies on the same side of the 
origin as the polar of $\ell=(-h(x):x_1:x_2)$. Clearly, the polar of $m$ 
intersects the interior of the set $W$ if and only if $x_0<-h(x)$ if and only 
if $(\frac{x_1}{x_0},\frac{x_2}{x_0})^T$ lies outside of $W^\ast$. The claim 
follows from setting $x_0=1$. 
\par
The claim extends to real lines in $(\P^2)^\ast$ through the origin, as their 
poles lie on the line at infinity. Having its pole at the origin, the line at 
infinity is also consistent with the claim.
\end{proof}
Under the assumptions of Lemma~\ref{Lem:WS}, it follows that an algebraic curve 
$D$ in $\P^2$ has a real tangent that intersects the interior of the set $W^\ast$, 
if and only if a real point of the dual curve $D^\vee$ lies outside of $W$. 
\par
\begin{Exa}\label{exa:television}
As an example, let us verify that the convex set
\[
S=\{(x_1,x_2)^T \mid x_1^6+x_2^6\leq 1\}
\]
is the dual convex set to $W$ defined in equation \eqref{eq:htele}.
We verified in Step~2 of Sec.~\ref{sec:revealingexample} that the 
point $(-h(x):x_1:x_2)\in\P^2$ is a root of $p=x_0^6-x_1^6-x_2^6$ for all 
nonzero $(x_1,x_2)^T\in\R^2$. Hence the tuples $\,-(x_1,x_2)^T/h(x)\,$ 
parametrize the boundary $x_1^6+x_2^6=1$ of $S$. We saw in the proof of 
Lemma~\ref{Lem:WS} that these tuples also parametrize the boundary of 
$W^\ast$, so we have $S=W^\ast$.
\par
We corroborate that the isolated, singular points of the real, affine part of 
$D^\vee$, depicted in Fig.~\ref{fig:television6}c), lie outside of $W$ 
because their polars, depicted in Fig.~\ref{fig:television6}d), intersect
the interior of the convex set $S$. 
\end{Exa}
%
%
%
\section{Aspects of the Proof of Kippenhahn's Assertion}
\label{sec:Kippenhahnright}
Following \cite[Thm.~4.9]{PSW2021}, we discuss aspects of a proof that 
the Kippenhahn curve $C_A$ is included in the numerical range $W=W(A)$. 
Here we use the determinant $p=\det(x_0\id+x_1K+x_2L)$ and matrices 
$A=K+\ii L$ as in Sec.~\ref{sec:constructKC}, whereas \cite{PSW2021} 
employs hyperbolic polynomials 
\cite{HeltonVinnikov2007,Netzer2012,PlaumannVinzant2013,Sinn2015}.
\par
\textbf{Observation 1.} 
The convex dual set to the numerical ranges is a well-known object. 
A set of the form 
\[
S=\{(x_1,x_2)^T\in\R^2: \,\id+x_1K+x_2L\,\text{ is positive semi-definite} \}
\]
is a \emph{spectrahedron}, see \cite{Netzer2012} and the references 
therein. This spectrahedron is the convex dual to the numerical range
\cite{HeltonSpitkovsky2012}, 
\begin{align*}
 S &= \{(x_1,x_2)^T\in\R^2: 
 \,\langle\eta,(\id+x_1K+x_2L)\eta\rangle\geq 0\,\, 
 \forall \eta\in\C^n, \langle \eta\mid \eta\rangle=1\,\}\\
   &= \{(x_1,x_2)^T\in\R^2: 
 \,1+x_1y_1+x_2y_2\geq 0\,\, 
 \forall (y_1,y_2)\in W \,\}\\
   &= W^\ast\,.
\end{align*}
\par
\textbf{Observation 2.} 
A real line that intersects an interior point of the spectrahedron $S$
intersects the curve $D$ only in real points. 
Let $\{e+\lambda x \mid \lambda\in\C\}$ be such a line, where $e=(e_1,e_2)^T$ 
is an interior point of $S$ and $x=(x_1,x_2)^T$ is nonzero. This line has the 
point at infinity $(0:x_1:x_2)$. The line intersects the curve $D$ at the points 
$e+\lambda x$, $\lambda\in\C$, for which 
\[
\det(\underbrace{\id+e_1K+e_2L}_A+\lambda\underbrace{(x_1K+x_2L)}_B)
=\det(A+\lambda B)
\]
is zero, and possibly at the point $(0:x_1:x_2)$ at infinity if $\det(B)=0$.
Since $e$ is an interior point of $S$, the matrix $A$ is positive definite.
Substituting $\lambda=-\frac{1}{\mu}$, we obtain
\[
\det(A+\lambda B)
=\det({\textstyle-\frac{1}{\mu}}A)
\det(-\mu\,\id+A^{-\frac{1}{2}}BA^{-\frac{1}{2}})\,.
\]
As the matrix $A^{-\frac{1}{2}}BA^{-\frac{1}{2}}$ is hermitian, its eigenvalues
are real. This shows that the line intersects $D$ only in real points.
\par
Roots at infinity can occur. For example, when $n=2$ and
\[
K=\left(\begin{array}{rr}
0 & 1\\1 & 0
\end{array}\right)
\qquad\text{and}\qquad
L=\left(\begin{array}{rr}
1 & 0\\0 & 0
\end{array}\right)\,,
\]
then $p=x_0^2+x_0x_2-x_1^2$. The affine part of the curve $D$ is the graph 
of the function $x_2=x_1^2-1$, which is a parabola.
The $x_2$-axis intersects the interior point $(0,0)^T$ of the spectrahedron 
enclosed by the parabola. The projective line corresponding to the $x_2$-axis
intersects the curve $D$ at the vertex of the parabola $(1:0:-1)$ and at the 
point $(0:0:1)$ at infinity. 
\par
\textbf{Observation 3.} 
We refer to \cite[Lemma 4.2]{PSW2021} and \cite[Sec.~6]{PSW2021}
for the statement that a real line that intersects the interior of the 
spectrahedron $S$ cannot be tangent to the curve $D$ at a real point.
This statement was also proved in \cite[Cor.~2.3]{ChienNakazato2010}.
\par
\textbf{The Proof.} 
Suppose the Kippenhahn curve $C_A$ is not included in the numerical range 
$W$\!. Then there is a real point $y$ outside of $W$ that lies on the dual 
curve $D^\vee$. By the definition of the dual curve, the polar of $y$ 
is a real tangent to the curve $D$ at some point $x$. 
Lemma~\ref{Lem:WS} shows this real tangent intersects the interior 
of the dual convex set $W^\ast$, which is $S=W^\ast$ by the first observation 
above. As per the second observation, the point $x$ is a real point. This 
contradicts the third observation.
\par
%
%
%
%
\section{Appendix: Equations of Curves}
\label{app:dualcurve}
The polynomial $p=\det(x_0\id+x_1K+x_2L)$ used in Sec.~\ref{sec:constructKC}, 
with matrices from equation \eqref{eq:KL}, is 
\[\begin{array}{l}
p = x_0^3
-\frac{3}{4} x_2 x_0^2
-2 x_1^2 x_0-\frac{21}{16} x_2^2 x_0
+\frac{55}{64} x_2^3
-\frac{3}{2} x_1^2 x_2\,.
\end{array}\]
The dual curve to the set $p=0$ is the set $q=0$, where $q$  
is the irreducible, homogeneous polynomial of degree six
\[\begin{array}{l}
1485 y_0^6
-3672 y_2 y_0^5
-15282 y_1^2 y_0^4
-2448 y_2^2 y_0^4
+12032 y_2^3 y_0^3
-19872 y_1^2 y_2 y_0^3\\
+9504 y_1^4 y_0^2
-5376 y_2^4 y_0^2
+21312 y_1^2 y_2^2 y_0^2
-6144 y_2^5 y_0
+13824 y_1^2 y_2^3 y_0\\
+27648 y_1^4 y_2 y_0
+864 y_1^6+4096 y_2^6
-4608 y_1^2 y_2^4
+13824 y_1^4 y_2^2\,.
\end{array}\]
The code from \texttt{Wolfram Mathematica} used for 
this calculation is as follows.
\vskip5pt\noindent
\hspace{1em}%
\includegraphics[width=\linewidth-1em]{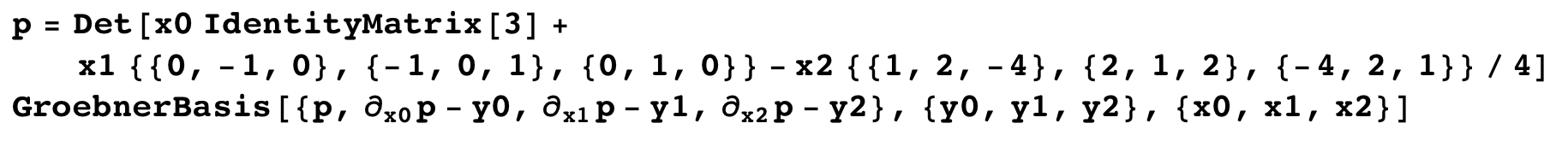}%
\par
The curve $x_0^6-x_1^6-x_2^6=0$ used in Sec.~\ref{sec:revealingexample}
is known as a \emph{Fermat curve} \cite[Sec.~3.6]{Fischer2001}. The dual 
curve is the set $q=0$, where $q$  is the irreducible, homogeneous 
polynomial of degree $30$
\[\begin{array}{l}
y_0^{30}
-5 y_1^6 y_0^{24}
-5 y_2^6 y_0^{24}
+10 y_1^{12} y_0^{18}
+10 y_2^{12} y_0^{18}
-605 y_1^6 y_2^6 y_0^{18}
-10 y_1^{18} y_0^{12}\\
-10 y_2^{18} y_0^{12}
-1905 y_1^6 y_2^{12} y_0^{12}
-1905 y_1^{12} y_2^6 y_0^{12}
+5 y_1^{24} y_0^6
+5 y_2^{24} y_0^6
-605 y_1^6 y_2^{18} y_0^6\\
+1905 y_1^{12} y_2^{12} y_0^6
-605 y_1^{18} y_2^6 y_0^6
-y_1^{30}
-y_2^{30}
-5 y_1^6 y_2^{24}
-10 y_1^{12} y_2^{18}
-10 y_1^{18} y_2^{12}\\
-5 y_1^{24} y_2^6\,.
\end{array}\]
The code from \texttt{Wolfram Mathematica} used for 
this calculation is as follows.
\vskip5pt\noindent
\hspace{1em}%
\includegraphics[width=\linewidth-1em]{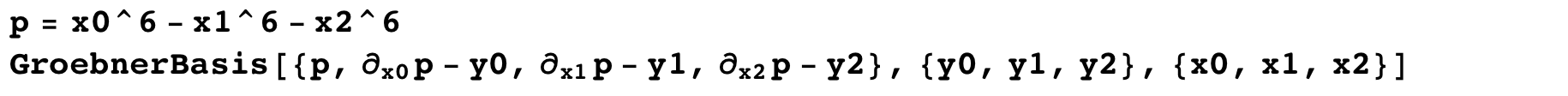}%
\par
%
%
\vskip\baselineskip\noindent
\textbf{Acknowledgments.} 
The author gratefully acknowledges the suggestions made by an anonymous referee, 
Ilya M.~Spitkovsky, and Daniel Plaumann, who helped to improve the paper. He is 
greatly indebted to Karol {\.Z}yczkowski for the hospitality during the IWOTA 
workshop in Kraków, Poland, in September 2022, and to the organizers of this 
workshop for the financial support.
\par
%
%
\bibliographystyle{plain}

\begin{thebibliography}{10}
%
\bibitem{Bebiano-etal2021} N.~Bebiano, J.~da Providéncia, and I.\,M.~Spitkovsky, 
\emph{On Kippenhahn curves and higher-rank numerical ranges of some matrices},
Linear Algebra and its Applications \textbf{629}, 246--257 (2021).
%
\bibitem{Blekerman-etal2013} G.~Blekherman, P.\,A.~Parrilo, and R.\,R.~Thomas, 
eds., \emph{Semidefinite Optimization and Convex Algebraic Geometry}, 
Philadelphia: SIAM, 2013.
%
\bibitem{ChienNakazato2010} M.-T.~Chien and H.~Nakazato, 
\emph{Joint numerical range and its generating hypersurface},
Linear Algebra and its Applications \textbf{432}:1, 173--179 (2010).
%
\bibitem{ChienNakazato2012} M.-T.~Chien and H.~Nakazato,
\emph{Singular points of the ternary polynomials associated with 4-by-4 matrices}, Electronic Journal of Linear Algebra \textbf{23}:1, 755--769 (2012).
%
\bibitem{Cox-etal2015} D.\,A.~Cox, J.~Little, and D.~O'Shea, 
\emph{Ideals, Varieties, and Algorithms}, 
Cham: Springer International Publishing, 2015.
%
\bibitem{Daepp-etal2018} U.~Daepp, P.~Gorkin, A.~Shaffer, and K.~Voss,
\emph{Finding Ellipses: What Blaschke Products, Poncelet's Theorem, 
and the Numerical Range Know About Each Other}, 
Providence: MAA Press, 2018.
%
\bibitem{Fischer2001} G.~Fischer, 
\emph{Plane Algebraic Curves}, 
Providence: AMS, 2001.
%
\bibitem{GauWu2021} H.-L.~Gau and P.\,Y.~Wu,
\emph{Numerical Ranges of Hilbert Space Operators},
Cambridge University Press, 2021.
%
\bibitem{Gawron-etal2010} P.~Gawron, Z.~Pucha{\l}a, J.\,A.~Miszczak, 
{\L}.~Skowronek, and K.~{\.Z}yczkowski,
\emph{Restricted numerical range: A versatile tool in the theory of quantum information},
Journal of Mathematical Physics \textbf{51}:10, 102204 (2010).
%
\bibitem{Hausdorff1919} F.~Hausdorff,
\emph{Der Wertvorrat einer Bilinearform},
Math.~Z.~\textbf{3}:1, 314--316 (1919).
%
\bibitem{HeltonSpitkovsky2012} J.\,W.~Helton and I.\,M.~Spitkovsky,
\emph{The possible shapes of numerical ranges},
Operators and Matrices \textbf{6}:3, 607--611 (2012). 
%
\bibitem{HeltonVinnikov2007} J.\,W.~Helton and V.~Vinnikov,
\emph{Linear matrix inequality representation of sets},
Communications on pure and applied mathematics \textbf{60}:5, 654--674 (2007).
%
\bibitem{Henrion2010} D.~Henrion,
\emph{Semidefinite geometry of the numerical range},
Electronic Journal of Linear Algebra \textbf{20}, 322--332 (2010).
%
\bibitem{HornJohnson1991} R.\,A.~Horn and C.\,R.~Johnson,
\emph{Topics in Matrix Analysis}, 
New York: Cambridge University Press, 1991.
%
\bibitem{Jefferies2004} B. Jefferies, 
\emph{Spectral Properties of Noncommuting Operators}, 
Berlin: Springer, 2004.
%
\bibitem{JiangSpitkovsky2022} M.~Jiang and I.\,M.~Spitkovsky,
\emph{On some reciprocal matrices with elliptical components of their Kippenhahn curves}, 
Special Matrices \textbf{10}:1, 117--130 (2022).
%
\bibitem{Keeler-etal1997} D.\,S.~Keeler, L.~Rodman, and I.\,M.~Spitkovsky,
\emph{The numerical range of $3\times 3$ matrices},
Linear Algebra and its Applications \textbf{252}:1--3, 115--139 (1997).
%
\bibitem{Kippenhahn1951} R.~Kippenhahn, 
\emph{Über den Wertevorrat einer Matrix},
Mathematische Nachrichten \textbf{6}:3--4, 193--228 (1951).
%
\bibitem{Mumford1976} D.~Mumford, 
\emph{Algebraic Geometry. 1: Complex Projective Varieties}, 
Corr.~2.~print, Berlin: Springer, 1976.
%
\bibitem{Netzer2012} T. Netzer,
\emph{Spectrahedra and Their Shadows},
Habilitationsschrift, Universität Leipzig, 2012.
%
\bibitem{PSW2021} D.~Plaumann, R.~Sinn, and S.~Weis, 
\emph{Kippenhahn's Theorem for joint numerical ranges and quantum states},
SIAM Journal on Applied Algebra and Geometry \textbf{5}:1, 86--113 (2021).
%
\bibitem{PlaumannVinzant2013} D.~Plaumann and C.~Vinzant,
\emph{Determinantal representations of hyperbolic plane curves: 
An elementary approach},
Journal of Symbolic Computation \textbf{57}, 48--60 (2013).
%
\bibitem{Pollack1974} F.\,M.~Pollack,
\emph{Numerical range and convex sets},
Can.\ math.\ bull.\ \textbf{17}:2, 295--296.
%
\bibitem{Sanyal-etal2011} R.~Sanyal, F.~Sottile, and B.~Sturmfels,
\emph{Orbitopes}, Mathematika \textbf{57}:02, 275--314 (2011).
%
\bibitem{Scheiderer2018} C.~Scheiderer,
\emph{Semidefinite representation for convex hulls of real algebraic curves},
SIAM Journal on Applied Algebra and Geometry \textbf{2}:1, 1--25 (2018).
%
\bibitem{Schneider2014} R.~Schneider, 
\emph{Convex Bodies: The Brunn-Minkowski Theory}, 2nd ed.,
New York: Cambridge University Press, 2014.
%
\bibitem{Sinn2015} R.~Sinn, 
\emph{Algebraic boundaries of convex semi-algebraic sets},
Mathematical Sciences (2015) 2:3.
\texttt{https://doi.org/10.1186/s40687-015-0022-0}
(open access)
%
\bibitem{Toeplitz1918} O.~Toeplitz,
\emph{Das algebraische Analogon zu einem Satze von Fej\'er},
Mathematische Zeitschrift \textbf{2}:1--2, 187--197 (1918).
%
%
\end{thebibliography}

%
%
%
\end{document}